\documentclass[11pt]{article}

\usepackage[hidelinks]{hyperref}
\hypersetup{
  colorlinks   = true, 
  urlcolor     = blue, 
  linkcolor    = blue, 
  citecolor   = red 
}
\usepackage{amsmath,amsthm,amssymb}
\usepackage{graphicx}
\usepackage{bbm}
\newcommand{\sy}{\boldsymbol{\Psi}}

\usepackage{xcolor}
\usepackage[margin=1in]{geometry} 
\usepackage{enumerate,mathtools,mathrsfs,bm,graphicx}
\usepackage[numbers, super]{natbib}
\usepackage[hidelinks]{hyperref}
\newcommand{\N}{\mathbb{N}}									

\newcommand{\R}{\mathbb{R}}

\newcommand{\vertiii}[1]{{\left\vert\kern-0.25ex\left\vert\kern-0.25ex\left\vert #1 
    \right\vert\kern-0.25ex\right\vert\kern-0.25ex\right\vert}}
    
\newcommand{\inner}[2]{\left\langle #1, #2 \right\rangle}
\newcommand{\norm}[1]{\left\Vert #1 \right\Vert}
\newcommand{\abs}[1]{\left\vert #1 \right\vert}

\newtheorem{theorem}{Theorem}[section]

\newtheorem{lemma}[theorem]{Lemma}
\newtheorem{proposition}[theorem]{Proposition}
\newtheorem*{remark}{Remark}
\newtheorem{definition}[theorem]{Definition}

\begin{document}
	\title{Infinite Boundary Friction Limit for Weak Solutions of the  Stochastic Navier-Stokes Equations}
	\author{Daniel Goodair \footnote{\'{E}cole Polytechnique F\'{e}d\'{e}rale de Lausanne, daniel.goodair@epfl.ch}}
	\date{\today} 
	\maketitle
\setcitestyle{numbers}	
\thispagestyle{empty}
\begin{abstract}
We address convergence of the unique weak solutions of the 2D stochastic Navier-Stokes equations with Navier boundary conditions, as the boundary friction is taken uniformly to infinity, to the unique weak solution under the no-slip condition. Our result is that for initial velocity in $L^2_x$, the convergence holds in probability in $C_tW^{-\varepsilon,2}_x \cap L^2_tL^2_x$ for any $0 < \varepsilon$. The noise is of transport-stretching type, although the theorem holds with other transport, multiplicative and additive noise structures. This seems to be the first work concerning the large boundary friction limit with noise, and convergence for weak solutions, due to only $L^2_{x}$ initial data, appears new even deterministically.

\end{abstract}
	
\tableofcontents
\textcolor{white}{Hello}
\thispagestyle{empty}
\newpage

\setcounter{page}{1}

\section{Introduction} \label{section introduction}

We are concerned with the 2D stochastic Navier-Stokes equations in a smooth bounded domain $\mathscr{O}$ under Navier boundary conditions, which read as
\begin{equation} \label{navier boundary conditions}
    u \cdot \underline{n} = 0, \qquad 2(Du)\underline{n} \cdot \mathbf{\underline{\iota}} + \alpha u\cdot \mathbf{\underline{\iota}} = 0
\end{equation} 
where $u$ represents the fluid velocity, $\underline{n}$ is the unit outwards normal vector, $\underline{\iota}$ the unit tangent vector, $Du$ is the rate of strain tensor $(Du)^{k,l}:=\frac{1}{2}\left(\partial_ku^l + \partial_lu^k\right)$ and $\alpha \in C^2(\partial \mathscr{O};\R)$ represents a friction coefficient which determines the extent to which the fluid slips on the boundary relative to the tangential stress. These conditions were first proposed by Navier in [\cite{navier1822memoire}, \cite{navier1827lois}], observed in [\cite{maxwell1879vii}] from the kinetic theory of gases and rigorously derived in [\cite{masmoudi2003boltzmann}] as a hydrodynamic limit of the Boltzmann equation under kinetic boundary conditions. Furthermore they have proven viable for modelling rough boundaries, whereby the irregular boundary is smoothened and supplemented with a homogenised boundary condition reflecting the average effect of the rough boundary; this avenue is explored in  [\cite{basson2008wall}, \cite{gerard2010relevance}, \cite{jager2001roughness}], and may well have taken motivation in earlier works such as [\cite{pare1992existence}] where the Navier boundary conditions are placed on an artificially imposed boundary within the boundary layer to capture the turbulent dynamics there, following the $k \sim \varepsilon$ model of turbulence introduced in [\cite{launder1983numerical}]. In this regime, for a smooth domain,  $\alpha$ is explicitly defined in terms of the velocity: otherwise, $\alpha$ depends on the roughness of the boundary. Experimental evidence supporting that boundary roughness dominates the dynamics is given in [\cite{zhu2002limits}].\\

Of course conditions (\ref{navier boundary conditions}) are less classical in the study of the Navier-Stokes equations than the no-slip condition, given simply by
\begin{equation} \label{noslip con}
    u = 0
\end{equation}
on the boundary $\partial \mathscr{O}$. Formally we can reconcile these notions by taking the boundary friction $\alpha$ to infinity in (\ref{navier boundary conditions}), as then $u \cdot \underline{\iota}$ dominates the second equation and the whole of (\ref{navier boundary conditions}) reduces to (\ref{noslip con}). The goal of this paper is to rigorously address the convergence of solutions under Navier boundary conditions (\ref{navier boundary conditions}) to the solution without slip at the boundary (\ref{noslip con}), in the presence of noise.

\subsection{The Model} \label{the model}

Specifically we shall consider the Navier-Stokes equations with a transport-stretching noise, following the principle of \textit{Stochastic Advection by Lie Transport} introduced in [\cite{holm2015variational}], given by
\begin{equation} \label{projected Ito Salt} 
    u_t = u_0 - \int_0^t\mathcal{P}\mathcal{L}_{u_s}u_s\ ds - \nu\int_0^t A u_s\, ds + \frac{1}{2}\int_0^t\sum_{i=1}^\infty \mathcal{P}B_i^2u_s ds - \int_0^t \mathcal{P}Bu_s d\mathcal{W}_s. 
\end{equation}
Here $\mathcal{P}$ denotes the Leray Projector onto divergence-free vector fields tangent to the boundary, $A = -\mathcal{P}\Delta$ is the Stokes Operator and $\mathcal{L}$ represents the nonlinear term defined for sufficiently regular functions $f,g:\mathscr{O} \rightarrow \R^2$ by $\mathcal{L}_fg:= \sum_{j=1}^2f^j\partial_jg$, where the superscript notation is for the $j^{\textnormal{th}}-$component mapping. As for the noise, $\mathcal{W}$ is a Cylindrical Brownian Motion and $B$ is defined with respect to a collection of divergence-free vector-fields $(\xi_i)$ by
\begin{equation} \label{def of trans stretch} B_if := \mathcal{L}_{\xi_i}f + \mathcal{T}_{\xi_i}f := \sum_{j=1}^2\left(\xi_i^j\partial_jf^j + f^j\nabla \xi_i^j \right).\end{equation}
The noise is introduced as a Stratonovich integral, at the non-projected level of the equation where the pressure appears; the $\sum_{i=1}^\infty\mathcal{P}B_i^2$ term arises as the It\^{o}-Stratonovich corrector after first projecting the equation by $\mathcal{P}$. We note use of the property that $\mathcal{P}B_i\mathcal{P}= \mathcal{P}B_i$ hence $(\mathcal{P}B_i)^2 = \mathcal{P}B_i^2$ observed in [\cite{goodair2022navier}] Lemma 2.7, and defer to [\cite{goodair2022navier}] Proposition 3.2 for the full conversion. Our motivation for considering this noise structure is threefold.\\

Firstly is the physical significance, which is entirely the motivation of [\cite{holm2015variational}]. In [\cite{holm2015variational}] the noise is derived through geometric variational principles and is shown to preserve Kelvin's Circulation Theorem. This theory has been expanded upon across [\cite{crisan2022variational}, \cite{holm2021stochastic}, \cite{street2021semi}], and has run in tandem with derivations of transport noise in fluids through a Lagrangian Reynolds Decomposition and Transport Theorem given by M\'{e}min [\cite{memin2014fluid}] and further developed in [\cite{chapron2018large}, \cite{debussche2024variational}, \cite{resseguier2017geophysical}]. The theory is bolstered by numerical analysis and data assimilation presented throughout [\cite{chapron2024stochastic}, \cite{cotter2020data}, \cite{cotter2019numerically},  \cite{crisan2023implementation}, \cite{crisan2023theoretical}, \cite{dufee2022stochastic}, \cite{ephrati2023data}] amongst many others. Stratonovich transport noise has also been derived following a stochastic model reduction scheme in [\cite{debussche2024second}, \cite{flandoli20212d}, \cite{flandoli2022additive}]. All of this recent progress supports the classical ideas of [\cite{brzezniak1992stochastic}, \cite{kraichnan1968small}, \cite{mikulevicius2001equations}, \cite{mikulevicius2004stochastic}], and we suggest [\cite{flandoli2023stochastic}] for a review of the topic.\\

Secondly is the fact that our choice of noise is very challenging, so proving this difficult case encompasses many others. Well-posedness results which are foundational for the deterministic Euler and Navier-Stokes equations remain open in this stochastic case, particularly in the presence of a boundary: in both the stochastic Euler and Navier-Stokes equations, for their most classical boundary condition, the existence of global strong solutions in 2D and local strong solutions in 3D is unknown. This is in contrast to a multiplicative noise which is (locally) bounded and Lipschitz on the relevant function spaces, where the corresponding theory is given in [\cite{glatt2009strong}, \cite{glatt2012local}]. Interestingly, however, the existence of strong solutions of (\ref{projected Ito Salt}) under Navier boundary conditions was proven in [\cite{goodair2025navier}], so the seemingly more complicated boundary condition acts favourably in the presence of noise. Fortunately the existence of weak solutions is known in both cases, which are the objects of this paper. We only discuss the open questions to motivate the analytical challenges in place with transport-type noise, and as a precursor to say that our results continue to hold with a bounded and Lipschitz multiplicative noise, additive noise in the projected $L^2_x$ space, and indeed a purely transport noise (without stretching term $\mathcal{T}_{\xi_i}$).\\

Thirdly we draw attention to the potential regularising properties of transport noise. This was brought to prominence in [\cite{flandoli2010well}] where Flandoli, Gubinelli and Priola prove that transport noise restores uniqueness in the transport equation. The theory was given a new lease of life from the work of Galeati [\cite{galeati2020convergence}], demonstrating that one can choose a transport noise such that the stochastic equation is close to the deterministic equation with added dissipation. In consequence the favourable properties of the deterministic equation with enhanced dissipation can be leveraged, leading in some cases to blow-up control
[\cite{agresti2024delayed}, \cite{flandoli2021delayed}, \cite{flandoli2021high}] and uniqueness [\cite{coghi2023existence}, \cite{galeati2023weak}]. We are keen to highlight this direction due to the relationship that the Navier boundary conditions have with the no-slip condition and the inviscid limit, along with the critical open problems regarding the inviscid limit and no-slip condition; we will expand on this in the following subsection.

\subsection{Deterministic Theory for the Navier Boundary Conditions} \label{deterministic theory}

Analytically, the Navier boundary conditions are attractive for three key reasons which we summarise below. Loosely speaking, let $u^{\alpha, \nu}$ denote a solution of the Navier-Stokes equations with Navier boundary friction $\alpha$ and viscosity $\nu > 0$, $u^{\infty,\nu}$ a solution of the Navier-Stokes equations with no-slip boundary condition and viscosity $\nu > 0$, and $u^{0}$ a solution of the Euler equation. Then, in a sense to be quantified, we have that:
\begin{enumerate}
    \item \label{first one in itemy} $u^{\alpha,\nu} \rightarrow u^{\infty,\nu}$ as $\alpha \rightarrow \infty$;
    \item \label{second one in itemy} $u^{\alpha,\nu} \rightarrow u^{0}$ as $\nu \rightarrow 0$;
    \item \label{third one in itemy} It is unknown whether $u^{\infty,\nu}$ converges to $u^0$ as $\nu \rightarrow 0$.
\end{enumerate}

Let us explain these properties. Item \ref{first one in itemy} is the direction of this paper, and was shown to hold in 2D in $C_tL^2_x \cap L^2_tW^{1,2}_x$ for initial $u_0 \in W^{3,2}_x$ in [\cite{kelliher2006navier}], which was relaxed to $u_0 \in W^{1,p}_x$ for $p > 2$ in [\cite{kim2009large}]. Item \ref{second one in itemy}, in contrast to item \ref{third one in itemy}, is perhaps the main analytical draw to the Navier boundary conditions. Item \ref{second one in itemy} has been shown in 2D and locally in 3D, in various topologies, throughout the works [\cite{clopeau1998vanishing}, \cite{filho2005inviscid}, \cite{iftimie2006inviscid}, \cite{kelliher2006navier}, \cite{masmoudi2012uniform}] and more. Meanwhile positive results for the inviscid limit under no-slip condition have been limited to very specific cases regarding analyticity of initial data or structure of the domain [\cite{lopes2008vanishing1}, \cite{lopes2008vanishing}, \cite{masmoudi1998euler}, \cite{sammartino998zero}, \cite{sammartino1998zero}], or conditional results such as [\cite{kato1984remarks}, \cite{kelliher2007kato} ,\cite{wang2001kato}] where the condition is not known to be satisfied for a general smooth domain with Sobolev valued initial condition. As supported observationally across [\cite{lyman1990vorticity}, \cite{morton1984generation}, \cite{serpelloni2013vertical}, \cite{tani1962production}, \cite{wallace2010measurement}], and seen explicitly in Kato's condition for convergence in [\cite{kato1984remarks}], the problem hinges upon controlling gradients of velocity along the boundary. The no-slip condition gives us no information to facilitate this control, however the Navier boundary conditions ensure that gradients are tractable at the boundary which ultimately enables one to deduce the convergence.\\

Item \ref{third one in itemy} is regarded as one of the most fundamental open problems in fluid mechanics. Based on the convergences in \ref{first one in itemy} and \ref{second one in itemy} it is natural to ask if the limits can be taken uniformly in $\nu$ and $\alpha$, respectively, to address \ref{third one in itemy}, though unsurprisingly there is a push-pull dynamic between them preventing it. This was considered in [\cite{kim2009large}]. It does, however, open up an alternative route to contemplate a regularisation by noise for \ref{third one in itemy}; whilst a first attempt at this problem would likely be to directly verify Kato's conditions for a well chosen noise, due to the poor interplay between noise and the no-slip condition a more promising path could be to regularise the solutions under Navier boundary conditions in a way which allows the limits of \ref{first one in itemy} and \ref{second one in itemy} to be taken uniformly in one another.

\subsection{Main Contributions and Relation to the Literature} \label{main contributions}

Informally, the main contribution of this paper is stated below. The complete result is Theorem \ref{main theorem}.
\begin{itemize}
    \item For given $u_0 \in L^2_{x}$, the unique weak solutions of the stochastic Navier-Stokes equations (\ref{projected Ito Salt}) with Navier boundary conditions (\ref{navier boundary conditions}) converge, as $\alpha \rightarrow \infty$ uniformly, to the unique weak solution under no-slip condition (\ref{noslip con}) in probability in $C_tW^{-\varepsilon,2}_x \cap L^2_tL^2_x$ for any $0 < \varepsilon$.
\end{itemize}

As far as we are aware, there are currently no results on the infinite boundary friction limit for stochastic Navier-Stokes equations. Furthermore, there do not appear to be any results treating the infinite boundary friction limit towards $\textit{weak}$ solutions of even the deterministic Navier-Stokes equations. As reviewed in the previous subsection, the only known results concern the limit with an initial condition $u_0 \in W^{1,p}_x$ with $p > 2$, where in fact $u_0$ is also required to be zero on the boundary; this condition seems to be omitted in the statement of Theorem 2 of [\cite{kim2009large}], but the author uses strong solutions under no-slip (e.g. Lemma 2) so it is clearly required. The condition is also explicitly included in [\cite{kelliher2006navier}], Theorem 9.2. The additional regularity of the strong solution under no-slip is necessary for their work. Relaxing to an $L^2_x$ initial condition comes with the trade-off of a weaker topology of convergence, namely $C_tW^{-\varepsilon,2}_x \cap L^2_tL^2_x$ as opposed to $C_tL^2_x \cap L^2_tW^{1,2}_x$, which we find to be valuable and hope that it adds insight to the problem even deterministically.\\

Nevertheless, it is worthwhile to ask whether or not we could recover the $C_tL^2_x \cap L^2_tW^{1,2}_x$ convergence under a smoother initial condition which is zero at the boundary. This is currently out of reach, intrinsically tied to the fact that the existence of strong solutions of (\ref{projected Ito Salt}) under the no-slip condition is open. The issue lies in the Leray projector failing to preserve the zero-trace property, so the Galerkin Projections related to the Dirichlet Stokes Operator may diverge on the noise in the energy norm of strong solutions; this is further discussed in the conclusion of [\cite{goodair2025navier}], along with how the Navier boundary conditions do not present the same problem. The method in [\cite{kelliher2006navier}, \cite{kim2009large}] is to directly look at the $L^2_x$ norm of the difference of the solutions. Strong solutions for the no-slip condition are immediately necessary in this approach, as weak solutions are only understood by testing with zero-trace vector fields which is failed by the solutions with Navier boundary conditions. Beyond this, the additional spatial regularity is explicitly required to control the difference.\\

Lacking the ability to control the norm of the difference of solutions directly, we instead argue by relative compactness which is a new approach to this problem. Uniform estimates in $C_tL^2_x \cap L^2_tW^{1,2}_x$, along with an additional control over small time increments, allow us to deduce tightness in $C_tW^{-\varepsilon,2}_x \cap L^2_tL^2_x$. Further uniform estimates of the boundary integral multiplied by $\alpha$ will imply that the solutions must approach zero at the boundary, from which a trace inequality presents that the limit we obtained from tightness must be zero-trace and is hence identifiable as the unique weak solution of the equation under no-slip condition. Whilst the tightness argument only gives us convergence in law of a subsequence \textit{a priori}, uniqueness of the limit ensures that the full sequence converges in probability.\\

Without the \textit{a priori} existence of strong solutions to (\ref{projected Ito Salt}) under the no-slip condition, one could still look to show convergence in $C_tL^2_x \cap L^2_tW^{1,2}_x$ by a tightness approach demonstrating uniform bounds in $C_tW^{1,2}_x \cap L^2_tW^{2,2}_x$. Such an approach is, however, tantamount to proving the existence of strong solutions under the no-slip condition. Estimates in $C_tW^{1,2}_x \cap L^2_tW^{2,2}_x$ of the solution $u^{\alpha}$ to (\ref{projected Ito Salt}) under Navier boundary conditions (\ref{navier boundary conditions}) yields a contribution from the noise of the form
$$\int_0^t\inner{\alpha \mathcal{P}B_iu_s^{\alpha}}{\mathcal{P}B_iu_s^{\alpha}}_{L^2(\partial \mathscr{O};\R^2)}ds$$
which we cannot handle as $\alpha \rightarrow \infty$. Once more we are undone by the nonlocal nature of the Leray Projector in failing to preserve the zero trace property, as without $\mathcal{P}$ we could force the $(\xi_i)$ to go to zero sufficiently quickly at the boundary and this term would vanish.\\

Overall, the stochastic Navier-Stokes equations with Navier boundary conditions have seen fairly little attention. In [\cite{brzezniak2001stochastic}] the authors prove the well-posedness and inviscid limit of linearised and regularised Navier-Stokes equations with a somewhat general multiplicative noise, though one which does not allow for gradient dependency, under a specific choice of Navier boundary conditions reducing to the requirement that vorticity is zero on the boundary: we shall call this the `free-boundary condition'. This followed the work of Bessaih in [\cite{bessaih1999martingale}] which proved the well-posedness and inviscid limit for the full Navier-Stokes system with a linear multiplicative noise under the free-boundary condition. The zero viscosity limit for the general Navier boundary conditions has thus far only been determined with additive noise, courtesy of [\cite{cipriano2015inviscid}]. For transport-stretching noise, well-posedness for any large $\alpha$ and the inviscid limit under free-boundary condition was established by the author in [\cite{goodair2025navier}]. Other stochastic fluid equations, with a bounded noise, under Navier boundary conditions have been considered in [\cite{tahraoui2024local}, \cite{tahraoui2025optimal}, \cite{tawri2024stochastic}].

\section{Preliminaries} \label{sec prelims}

Ahead of the main proofs, we overview some preliminaries which will allow us to concretely state the main result and underpin the analysis moving forward. Subsection \ref{subs functional anal} introduces the relevant function spaces and key properties from the deterministic theory. Stochastic integration is recapped in Subsection \ref{subs stoch prelim} followed by the transport-stretching noise in Subsection \ref{subs transport stretch}. The key definitions and results can then be given in Subsection \ref{subs main results}.

\subsection{Functional Analytic Preliminaries} \label{subs functional anal}

This subsection is essentially a shortening of the more extended summary [\cite{goodair2025navier}] Subsection 1.2, and for even further details we suggest [\cite{acevedo2019stokes}, \cite{clopeau1998vanishing}, \cite{kelliher2006navier}] where many of these preliminary results are proven.\\

Let us first address the divergence-free and boundary conditions, which will be imposed on our solution through the function spaces in which it takes value. Recall that the divergence of a function $f \in W^{1,2}(\mathscr{O};\R^2)$ is defined by $\textnormal{div}f := \sum_{j=1}^2 \partial_jf^j$, and $f$ is said to be divergence-free if $\textnormal{div}f =0$. The rate of strain tensor $D$ appearing in (\ref{navier boundary conditions}) is a mapping $D: W^{1,2}(\mathscr{O};\R^2) \rightarrow L^{2}(\mathscr{O};\R^{2 \times 2})$ defined by $$f \mapsto \begin{bmatrix}
        \partial_1f^1 & \frac{1}{2}\left(\partial_1f^1 + \partial_2f^2\right)\\
        \frac{1}{2}\left(\partial_1f^1 + \partial_2f^2\right) & \partial_2f^2
    \end{bmatrix}
    $$
    or in component form, $(Df)^{k,l}:=\frac{1}{2}\left(\partial_kf^l + \partial_lf^k\right)$. 
\begin{definition}
We define $C^{\infty}_{0,\sigma}(\mathscr{O};\R^2)$ as the subset of $C^{\infty}(\mathscr{O};\R^2)$ of functions which are compactly supported and divergence-free. $L^2_\sigma$ is defined as the completion of $C^{\infty}_{0,\sigma}(\mathscr{O};\R^2)$ in $L^2(\mathscr{O};\R^2)$, and $W^{1,2}_{\sigma}$ as the completion of $C^{\infty}_{0,\sigma}(\mathscr{O};\R^2)$ in $W^{1,2}(\mathscr{O};\R^2)$. Moreover we introduce $\bar{W}^{1,2}_\sigma$ as the intersection of $W^{1,2}(\mathscr{O};\R^2)$ with $L^2_\sigma$ and $\bar{W}^{2,2}_{\alpha}$ by $$\bar{W}^{2,2}_{\alpha}:= \left\{f \in W^{2,2}(\mathscr{O};\R^2) \cap \bar{W}^{1,2}_{\sigma}: 2(Df)\underline{n} \cdot \underline{\iota} + \alpha f \cdot \underline{\iota} = 0 \textnormal{ on } \partial \mathscr{O}\right\}.$$
\end{definition}

\begin{remark} \label{new first labelled remark}
    $L^2_{\sigma}$ can be characterised as the subspace of $L^2(\mathscr{O};\R^2)$ of weakly divergence-free functions with normal component weakly zero at the boundary (see [\cite{robinson2016three}] Lemma 2.14). $W^{1,2}_{\sigma}$ is precisely the subspace of $W^{1,2}(\mathscr{O};\R^2)$ consisting of divergence-free and zero-trace functions (see [\cite{temam2001navier}] Theorem 1.6). 
\end{remark}


The Leray Projector $\mathcal{P}$ stated in the introduction can now be properly defined as the orthogonal projection in $L^2\left(\mathscr{O};\R^2\right)$ onto $L^2_{\sigma}$. We note that the Poincar\'{e} Inequality holds for the component mappings of functions in $\bar{W}^{1,2}_{\sigma}$, of course also true for $W^{1,2}_{\sigma}$, so we equip both spaces with the inner product $$\inner{f}{g}_1 := \sum_{j=1}^2 \inner{\partial_j f}{\partial_j g}$$ which is equivalent to the full $W^{1,2}(\mathscr{O};\R^2)$ inner product. Defining $\kappa \in C^2(\partial \mathscr{O};\R)$ to be the curvature of $\partial \mathscr{O}$, then as demonstrated in [\cite{kelliher2006navier}] equation (5.1) we have the important Green's type identity that for all $f \in \bar{W}^{2,2}_{\alpha}$, $\phi \in \bar{W}^{1,2}_{\sigma}$, \begin{equation} \label{eq greens for navier} \inner{\Delta f}{\phi}  = -\inner{f}{\phi}_1 + \inner{(\kappa - \alpha)f}{\phi}_{L^2(\partial \mathscr{O}; \R^2)}.\end{equation}
Here and throughout the text $\inner{\cdot}{\cdot}$ denotes $\inner{\cdot}{\cdot}_{L^2(\mathscr{O};\R^2)}$ and in general we will represent $W^{s,p}(\mathscr{O};\R^2)$ by $W^{s,p}$ and $\norm{\cdot}_{W^{s,p}(\mathscr{O};\R^2)}$ by $\norm{\cdot}_{W^{s,p}}$. To control the boundary integral we will make use of the classical inequality \begin{equation}
    \label{inequality from Lions}
    \norm{f}_{L^2(\partial \mathscr{O};\R^2)}^2 \leq c\norm{f}\norm{f}_{W^{1,2}},
\end{equation}
see for instance [\cite{lions1996mathematical}] pp.130, [\cite{kelliher2006navier}] equation (2.5). Returning to the nonlinear term, we shall frequently understand the $L^2$ inner product as a duality pairing between $L^{\frac{4}{3}}$ and $L^4$, justified by the observation that for every $\phi,f,g, \in W^{1,2}$,
  \begin{align} \nonumber
       \left\vert\inner{\mathcal{L}_{\phi}f}{g}\right\vert \leq \norm{\mathcal{L}_{\phi}f}_{L^{4/3}}\norm{g}_{L^4} &\leq c\sum_{k=1}^2\norm{\phi}_{L^4}\norm{\partial_kf}\norm{g}_{L^4} \leq c\norm{\phi}^{\frac{1}{2}}\norm{\phi}_1^{\frac{1}{2}}\norm{f}_1\norm{g}^{\frac{1}{2}}\norm{g}_1^{\frac{1}{2}}\label{a bound in align}
    \end{align}
having used H\"{o}lder's and Ladyzhenskaya's Inequalities. We also appreciate that the usual anti-symmetry property continues to hold for transport along vector fields which are divergence-free and only tangential to the boundary; that is for every $\phi \in \bar{W}^{1,2}_{\sigma}$, $f,g \in W^{1,2}$ we have that \begin{equation}\label{wloglhs}\inner{\mathcal{L}_{\phi}f}{g}= -\inner{f}{\mathcal{L}_{\phi}g}.\end{equation}
Now let us move on to the Stokes Operator $A$. Weak solutions of (\ref{projected Ito Salt}) are only required to belong to $W^{1,2}$, and given the derivative operator $D$ in the Navier Boundary Conditions (\ref{navier boundary conditions}) then such solutions are not regular enough to make sense of the boundary condition in terms of the trace. Therefore the boundary condition is not enforced in the regularity of the solution, but rather in the definition of the Stokes Operator on the solution space. More precisely whilst $A=-\mathcal{P}\Delta$ is universally defined on $W^{2,2}$, its extension to $\bar{W}^{1,2}_{\sigma}$ is greatly dependent on the dense domain of definition that one chooses to extend from. We recall from [\cite{clopeau1998vanishing}] Lemma 2.2 that $\bar{W}^{2,2}_{\alpha}$ is dense in $\bar{W}^{1,2}_{\sigma}$, and considering the Gelfand Triple
$$\bar{W}^{1,2}_{\sigma} \xhookrightarrow{} L^2_{\sigma} \xhookrightarrow{} \left(\bar{W}^{1,2}_{\sigma}\right)^*$$
we define $A_{\alpha}:\bar{W}^{1,2}_{\sigma} \rightarrow \left(\bar{W}^{1,2}_{\sigma}\right)^*$  by 
$$\inner{A_{\alpha}f}{\phi}_{\left(\bar{W}^{1,2}_{\sigma}\right)^* \times \bar{W}^{1,2}_{\sigma}} = \inner{f}{\phi}_1 - \inner{(\kappa - \alpha)f}{\phi}_{L^2(\partial \mathscr{O}; \R^2)}.$$
Due to (\ref{eq greens for navier}) if $f \in \bar{W}^{2,2}_{\alpha}$ then $A_{\alpha}f = Af$ so $A_{\alpha}$ is truly the unique continuous extension of $A$ from the dense subspace $\bar{W}^{2,2}_{\alpha}$. Similarly we define $A_{\infty} : W^{1,2}_{\sigma} \rightarrow \left(W^{1,2}_{\sigma}\right)^*$ by the more straightforward
$$\inner{A_{\infty}f}{\phi}_{\left(\bar{W}^{1,2}_{\sigma}\right)^* \times \bar{W}^{1,2}_{\sigma}} = \inner{f}{\phi}_1$$
which, due to the zero-trace condition, extends $A$ from $W^{1,2}_{\sigma} \cap W^{2,2}$. To obtain optimal convergence results we consider fractional spaces defined through the spectrum of $A_{\infty}$; as a consequence of the Rellich-Kondrachov Theorem, see for instance [\cite{robinson2016three}] Theorem 2.24, there exists a collection of functions $(a_k)$, $a_k \in W^{1,2}_{\sigma} \cap C^{\infty}\left(\bar{\mathscr{O}};\R^2\right)$ such that the $(a_k)$ are eigenfunctions of $A$, are an orthonormal basis in $L^2_{\sigma}$ and with corresponding eigenvalues $(\lambda_k)$ strictly positive and approaching infinity as $k \rightarrow \infty$.\footnote{Note that we could also find a basis of eigenfunctions belonging to $\bar{W}^{2,2}_{\alpha}$, see [\cite{clopeau1998vanishing}] Lemma 2.2, which would define the spectrum of $A_{\alpha}$.} For $0 \leq \varepsilon \leq 1$ we define
$$W^{\varepsilon,2}_{\sigma}:= \left\{f \in L^2_{\sigma}: \sum_{k=1}^\infty\lambda_k^{\varepsilon}\inner{f}{a_k}^2 < \infty \right\}.$$
This space does indeed coincide with $L^2_{\sigma}$, $W^{1,2}_{\sigma}$ for $\varepsilon = 0,1$ respectively, and for every $\varepsilon >0$ we have that $W^{\varepsilon,2}_{\sigma}$ compactly embeds into $L^2_{\sigma}$. Moreover in the context of the Gelfand Triple
$$W^{\varepsilon,2}_{\sigma} \xhookrightarrow{} L^2_{\sigma} \xhookrightarrow{} \left(W^{\varepsilon,2}_{\sigma}\right)^*$$
 we define $W^{-\varepsilon,2}_{\sigma}:= \left(W^{\varepsilon,2}_{\sigma}\right)^*.$

\subsection{Stochastic Preliminaries} \label{subs stoch prelim}

Let $(\Omega,\mathcal{F},(\mathcal{F}_t), \mathbbm{P})$ be a fixed filtered probability space satisfying the usual conditions of completeness and right continuity. We take $\mathcal{W}$ to be a cylindrical Brownian motion over some Hilbert Space $\mathfrak{U}$ with orthonormal basis $(e_i)$. Recall (e.g. [\cite{lototsky2017stochastic}], Definition 3.2.36) that $\mathcal{W}$ admits the representation $\mathcal{W}_t = \sum_{i=1}^\infty e_iW^i_t$ as a limit in $L^2(\Omega;\mathfrak{U}')$ whereby the $(W^i)$ are a collection of i.i.d. standard real valued Brownian Motions and $\mathfrak{U}'$ is an enlargement of the Hilbert Space $\mathfrak{U}$ such that the embedding $J: \mathfrak{U} \rightarrow \mathfrak{U}'$ is Hilbert-Schmidt and $\mathcal{W}$ is a $JJ^*-$cylindrical Brownian Motion over $\mathfrak{U}'$. Given a process $F:[0,T] \times \Omega \rightarrow \mathscr{L}^2(\mathfrak{U};\mathscr{H})$ progressively measurable and such that $F \in L^2\left(\Omega \times [0,T];\mathscr{L}^2(\mathfrak{U};\mathscr{H})\right)$, for any $0 \leq t \leq T$ we define the stochastic integral $$\int_0^tF_sd\mathcal{W}_s:=\sum_{i=1}^\infty \int_0^tF_s(e_i)dW^i_s,$$ where the infinite sum is taken in $L^2(\Omega;\mathscr{H})$. We can extend this notion to processes $F$ which are such that $F(\omega) \in L^2\left( [0,T];\mathscr{L}^2(\mathfrak{U};\mathscr{H})\right)$ for $\mathbbm{P}-a.e.$ $\omega$ via the traditional localisation procedure. In this case the stochastic integral is a local martingale in $\mathscr{H}$. We defer to [\cite{goodair2024stochastic}] Chapter 2 for further details on this construction and properties of the stochastic integral. The stochastic integral of (\ref{projected Ito Salt}) is then understood by $\mathcal{P}Bu_s(e_i) = \mathcal{P}B_i(u_s)$, see [\cite{goodair2024stochastic}] Subchapter 3.2. We shall make frequent use of the Burkholder-Davis-Gundy Inequality ([\cite{da2014stochastic}] Theorem 4.36) and the energy identity (Proposition \ref{Ito formula}).

\subsection{Transport-Stretching Noise} \label{subs transport stretch}

We collect some fundamental properties of the transport-stretching noise $B_i$, as defined in (\ref{def of trans stretch}). We assume that $\xi_i \in L^2_{\sigma} \cap W^{2,\infty}$ with $\sum_{i=1}^\infty \norm{\xi_i}_{W^{2,\infty}}^2 < \infty$. The following properties are taken from  [\cite{goodair2022navier}] Subsection 2.3, where a more complete description is deferred to. Firstly for $k = 0, 1$ and $f \in W^{1,2}$ there exists a constant $c$ such that  
\begin{align}
    \label{T_ibound}\norm{\mathcal{T}_{\xi_i}f}_{W^{k,2}}^2 &\leq  c \norm{\xi_i}^2_{W^{k+1,\infty}}\norm{f}^2_{W^{k,2}}\\
    \label{L_ibound} \norm{\mathcal{L}_{\xi_i}f}_{W^{k,2}}^2 &\leq c\norm{\xi_i}^2_{W^{k,\infty}}\norm{f}^2_{W^{k+1,2}}\\
    \label{boundsonB_i} \norm{B_if}_{W^{k,2}}^2 &\leq c\norm{\xi_i}^2_{W^{k+1,\infty}}\norm{f}^2_{W^{k+1,2}}
\end{align}
Moreover $\mathcal{T}_{\xi_i}$ is a bounded linear operator on $L^2$ so has adjoint $\mathcal{T}_{\xi_i}^*$ satisfying the same boundedness. $\mathcal{L}_{\xi_i}$ is a densely defined operator in $L^2$ with domain of definition $W^{1,2}$, and has adjoint $\mathcal{L}_{\xi_i}^*$ in this space given by $-\mathcal{L}_{\xi_i}$ with same dense domain of definition, due to (\ref{wloglhs}). Likewise then $B_i^*$ is the densely defined adjoint $-\mathcal{L}_{\xi_i} + \mathcal{T}_{\xi_i}^*$. In addition, [\cite{goodair2022navier}] Proposition 2.6 shows that there exists a constant $c$ such that for all $f \in W^{1,2}$,
\begin{align} \label{noise bound 1}\inner{B_if}{B_i^*f} + \norm{B_if}^2 &\leq c\norm{\xi_i}_{W^{2,\infty}}^2\norm{f}^2,\\
\label{noise bound 2} \inner{B_if}{f}^2 &\leq c\norm{\xi_i}_{W^{1,\infty}}^2\norm{f}^4.
\end{align}
Similarly to $A_{\alpha}$, $\mathcal{P}B_i^2$ continuously extends to an operator from $\bar{W}^{1,2}_{\sigma}$ to $\left( \bar{W}^{1,2}_{\sigma}\right)^*$ by
$$\inner{\mathcal{P}B_i^2f}{\phi}_{\left( \bar{W}^{1,2}_{\sigma}\right)^* \times \bar{W}^{1,2}_{\sigma}} = \inner{B_if}{B_i^*\phi}.$$

\subsection{Definitions and Results} \label{subs main results}

As discussed in Subsection \ref{subs functional anal}, the definition of weak solutions is dependent on the extension of the Stokes Operator $A$ to $A_{\alpha}$, $A_{\infty}$. To this end, we consider (\ref{projected Ito Salt}) with Navier friction $\alpha_n$ by
\begin{equation} \label{NSalphan} \tag{$NS\alpha_n$}
    u^n_t = u_0 - \int_0^t\mathcal{P}\mathcal{L}_{u^n_s}u^n_s\ ds - \nu\int_0^t A_{\alpha_n} u^n_s\, ds + \frac{1}{2}\int_0^t\sum_{i=1}^\infty \mathcal{P}B_i^2u^n_s ds - \int_0^t \mathcal{P}Bu^n_s d\mathcal{W}_s
\end{equation}
and with no-slip condition by

\begin{equation} \label{NSinfty} \tag{$NS\infty$}
    u_t = u_0 - \int_0^t\mathcal{P}\mathcal{L}_{u_s}u_s\ ds - \nu\int_0^t A_{\infty} u_s\, ds + \frac{1}{2}\int_0^t\sum_{i=1}^\infty \mathcal{P}B_i^2u_s ds - \int_0^t \mathcal{P}Bu_s d\mathcal{W}_s.
\end{equation}
The index by $n$ in (\ref{NSalphan}) is in place to consider a sequence $\alpha_n \rightarrow \infty$. To state the main theorem we first settle the solution theory, which is given for any arbitrary but fixed time horizon $T > 0$.

\begin{definition} \label{def solution for NBC}
    Let  $u_0: \Omega \rightarrow L^2_{\sigma}$ be $\mathcal{F}_0-$measurable and $\alpha_n \in C^2(\partial \mathscr{O}; \R)$. A process $u^n$ which is progressively measurable in $\bar{W}^{1,2}_{\sigma}$ such that for $\mathbbm{P}-a.e.$ $\omega$, $u^n_{\cdot}(\omega) \in C\left([0,T];L^2_{\sigma}\right) \cap L^2\left([0,T];\bar{W}^{1,2}_{\sigma}\right)$, is said to be a weak solution of the equation (\ref{NSalphan}) if it satisfies the identity
\begin{align} \nonumber
     \inner{u^n_t}{\psi} = \inner{u_0}{\psi} - \int_0^{t}\inner{\mathcal{L}_{u^n_s}u^n_s}{\psi}ds &- \nu\int_0^{t} \inner{u^n_s}{\psi}_1 ds + \nu\int_0^t \inner{(\kappa - \alpha_n)u^n_s}{\psi}_{L^2(\partial \mathscr{O}; \R^2)}ds\\ &+ \frac{1}{2}\int_0^{t}\sum_{i=1}^\infty \inner{B_iu^n_s}{B_i^*\psi} ds - \int_0^{t} \inner{Bu^n_s}{\psi} d\mathcal{W}_s\label{identityindefinitionofspatiallyweak}
\end{align}
for every $\psi \in \bar{W}^{1,2}_{\sigma}$, $\mathbbm{P}-a.s.$ in $\R$ for all $t\in[0,T]$. Moreover $u^n$ is said to be unique if for any other such process $v^n$, $u^n$ and $v^n$ are indistinguishable.
\end{definition}

\begin{remark}
   As established, satisfaction of (\ref{identityindefinitionofspatiallyweak}) is equivalent to (\ref{NSalphan}) holding in $\left(\bar{W}^{1,2}_{\sigma}\right)^*$.
\end{remark}


\begin{definition} \label{def no slip}
        Let  $u_0: \Omega \rightarrow L^2_{\sigma}$ be $\mathcal{F}_0-$measurable. A process $u$ which is progressively measurable in $W^{1,2}_{\sigma}$ such that for $\mathbbm{P}-a.e.$ $\omega$, $u_{\cdot}(\omega) \in C\left([0,T];L^2_{\sigma}\right) \cap L^2\left([0,T];W^{1,2}_{\sigma}\right)$, is said to be a weak solution of the equation (\ref{NSinfty}) if it satisfies the identity
\begin{align} \nonumber
     \inner{u_t}{\phi} = \inner{u_0}{\phi} - \int_0^{t}\inner{\mathcal{L}_{u_s}u_s}{\phi}ds &- \nu\int_0^{t} \inner{u_s}{\phi}_1 ds\\ &+ \frac{1}{2}\int_0^{t}\sum_{i=1}^\infty \inner{B_iu_s}{B_i^*\phi} ds - \int_0^{t} \inner{Bu_s}{\phi} d\mathcal{W}_s\label{identityindefinitionofspatiallyweaknoslip}
\end{align}
for every $\phi \in W^{1,2}_{\sigma}$, $\mathbbm{P}-a.s.$ in $\R$ for all $t\in[0,T]$. Moreover $u$ is said to be unique if for any other such process $v$, $u$ and $v$ are indistinguishable.
\end{definition}

Similarly to the previous remark, satisfaction of (\ref{identityindefinitionofspatiallyweaknoslip}) is equivalent to (\ref{NSinfty}) holding in $W^{-1,2}_{\sigma}$. 

\begin{proposition} \label{existence results for both}
    Let  $u_0: \Omega \rightarrow L^2_{\sigma}$ be $\mathcal{F}_0-$measurable. Then there exists a unique weak solution of the equation (\ref{NSinfty}). Furthermore for any given $\alpha_n \in C^2(\partial \mathscr{O}; \R)$, there exists a unique weak solution of the equation (\ref{NSalphan}). 
\end{proposition}

\begin{proof}
    Existence and uniqueness of weak solutions for (\ref{NSinfty}) was proven in [\cite{goodair2023zero}] Theorem 1.10, and for (\ref{NSalphan}) in [\cite{goodair2025navier}] Theorem 1.14. The only adjustment to these results is that fact that we have only assumed each $\xi_i \in \bar{W}^{1,2}_{\sigma}$ instead of $W^{1,2}_{\sigma}$, that is $\xi_i$ need not be zero-trace. Whilst a rapid decay to zero at the boundary is necessary for the existence of \textit{strong} solutions to (\ref{NSalphan}) in [\cite{goodair2025navier}], due to (\ref{wloglhs}) and the resultant bounds (\ref{noise bound 1}), (\ref{noise bound 2}), the zero-trace condition on $\xi_i$ is actually unnecessary for the weak solution theory in both cases and Proposition \ref{existence results for both} holds.
\end{proof}

We now state the main theorem.

\begin{theorem} \label{main theorem}
    Let $u_0 \in L^2\left(\Omega; L^2_{\sigma} \right)$ be $\mathcal{F}_0-$measurable, and $(\alpha_n)$ a sequence in $C^2\left(\partial \mathscr{O};\R \right)$ which is such that $\inf_{x \in \partial \mathscr{O}} \alpha_n(x)$ diverges to infinity and for sufficiently large $N$, $$\sup_{n \geq N}\left[\frac{\sup_{\partial \mathscr{O}} \alpha_n}{\inf_{\partial \mathscr{O}} \alpha_n} \right] < \infty. $$ Let $(u^n)$ be the corresponding sequence of unique weak solutions to (\ref{NSalphan}), and $u$ the unique weak solution to (\ref{NSinfty}). Then for any $0 < \varepsilon \leq 1$, $(u^n)$ converges to $u$ in probability in $C\left([0,T];W^{-\varepsilon,2}_{\sigma} \right) \cap L^2\left([0,T];L^2_{\sigma}\right)$. 
\end{theorem}

\begin{remark}
    As a straightforward example, one could define the sequence of frictions by $\alpha_n := n$ everywhere on $\partial \mathscr{O}$.
\end{remark}

\section{Infinite Boundary Friction Limit} \label{section infinite boundary}

This section is devoted to the proof of Theorem \ref{main theorem}. Towards this goal we fix the assumed elements of Theorem \ref{main theorem}, namely the initial condition $u_0 \in L^2\left(\Omega;L^2_{\sigma}\right)$, the parameter $0 < \varepsilon \leq 1$ and the sequence $(\alpha_n)$. As the curvature $\kappa \in C^2(\partial \mathscr{O};\R)$ then it is bounded, so for simplicity in the following we shift the sequence $(\alpha_n)$ to consider only sufficiently large $n$ such that $0 \vee (\kappa + 1) < \alpha_n$ everywhere on $\partial \mathscr{O}$ and
\begin{equation} \label{sup bound} \sup_{n \in \N}\left[\frac{\sup_{\partial \mathscr{O}} \alpha_n}{\inf_{\partial \mathscr{O}} \alpha_n} \right] < \infty.\end{equation}
In particular,
\begin{equation} \label{sup bound with kappa} \sup_{n \in \N}\left[\frac{\sup_{\partial \mathscr{O}} (\alpha_n - \kappa)}{\inf_{\partial \mathscr{O}} (\alpha_n - \kappa)} \right] < \infty.\end{equation}
For various estimates in this section we will use the stopping times defined for $R > 0$ by
\begin{equation}  \tau^R_n:= T \wedge \inf\left\{s \geq 0: \sup_{r \in [0,s]}\norm{u^n_r}^2 + \int_0^{s}\norm{u^n_r}_{1}^2 dr + \left[\sup_{\partial \mathscr{O}}(\alpha_n - \kappa)\right]\int_0^s \norm{u^n_r}_{L^2(\partial \mathscr{O};\R^2)}^2dr \geq R \right\} 
\nonumber
\end{equation}
along with notation
\begin{equation} \label{check notation} \check{u}^n_{\cdot}:= u^n_{\cdot}\mathbbm{1}_{\cdot \leq \tau^R_n}, \qquad \hat{u}^n_{\cdot}:= u^n_{\cdot \wedge \tau^R_n}. \end{equation}
At first, in Subsection \ref{subs uniform estimates}, we shall prove uniform estimates on $(u^n)$ in the energy norm of weak solutions. This is followed by demonstrating that these solutions approach zero at the boundary in Subsection \ref{subs convergence to zero}, with tightness in $C\left([0,T];W^{-\varepsilon,2}_{\sigma} \right) \cap L^2\left([0,T];L^2_{\sigma}\right)$ the content of Subsection \ref{subs tightness}. We identify the limit obtained from tightness in Subsection \ref{subs identification of the limit} and conclude the proof in Subsection \ref{subs final steps}.

\subsection{Uniform Interior Estimates} \label{subs uniform estimates}

\begin{lemma} \label{lemma uniform}
    There exists a constant $C$ such that for all $n \in \N$,
    $$\mathbbm{E}\left(\sup_{r \in [0,T]}\norm{u^n_r}^2 + \int_0^T\norm{u^n_s}_1^2ds \right) \leq C\mathbbm{E}\left(\norm{u_0}^2 \right).$$
    \end{lemma}

\begin{proof}
  We appreciate that use of the stopping time $\tau^R_n$ allows us to rigorously take expectations in verifying this estimate. Stopping $u^n$ at $\tau^R_n$ and applying the It\^{o} Formula, we deduce the energy equality
\begin{align}\nonumber\norm{u^{n}_{ r \wedge \tau^R_n}}^2 &= \norm{u_0}^2 - 2\int_0^{r\wedge \tau^R_n}\inner{\mathcal{L}_{u^n_s}u^n_s}{u^n_s}ds\\ \nonumber &- 2\nu\int_0^{r\wedge \tau^R_n}\norm{u^n_s}_1^2ds + 2\nu\int_0^{r\wedge \tau^R_n}\inner{(\kappa - \alpha_n)u^n_s}{u^n_s}_{L^2(\partial \mathscr{O};\R^2)}ds\\ \nonumber &+ \int_0^{r\wedge \tau^R_n} \sum_{i=1}^\infty \inner{B_iu^n_s}{B_i^*u^n_s}ds + \int_0^{r\wedge\tau^R_n}\sum_{i=1}^\infty\norm{\mathcal{P}B_iu^n_s}^2ds\\ &- 2\sum_{i=1}^\infty\int_0^{r\wedge\tau^R_n}\inner{B_iu^n_s}{u^n_s}dW^i_s \label{for the uniform estimates}
\end{align}
for every $r \in [0,T]$, $\mathbbm{P}-a.s.$. Using the cancellation in the nonlinear term, that $\kappa - \alpha_n < 0$ and the noise estimate (\ref{noise bound 1}), we quickly reduce to the inequality
\begin{align}\nonumber\norm{u^{n}_{ r \wedge \tau^R_n}}^2 + 2\nu\int_0^{r\wedge \tau^R_n}\norm{u^n_s}_1^2ds = \norm{u_0}^2 + c\int_0^{r\wedge \tau^R_n} \norm{u^n_s}^2 ds - 2\sum_{i=1}^\infty\int_0^{r\wedge\tau^R_n}\inner{B_iu^n_s}{u^n_s}dW^i_s. 
\end{align}
Note that to apply the noise estimate (\ref{noise bound 1}) we have first used that $\norm{\mathcal{P}B_iu^n_s}^2 \leq \norm{B_iu^n_s}^2$ given that $\mathcal{P}$ is an orthogonal projection. Now we bound by the absolute value, take the supremum over $r \in [0,t]$ followed by the Burkholder-Davis-Gundy Inequality, as well as plugging in the notation (\ref{check notation}) to see that
\begin{align}\nonumber\mathbbm{E}\left(\sup_{r \in [0,t]}\norm{\check{u}^{n}_{ r }}^2 + 2\nu\int_0^{t}\norm{\check{u}^n_s}_1^2ds\right) \leq c\mathbbm{E}\left[\norm{u_0}^2 + \int_0^{t} \norm{\check{u}^n_s}^2 ds + \left(\int_0^{t}\sum_{i=1}^\infty\inner{B_i\check{u}^n_s}{\check{u}^n_s}^2ds\right)^{\frac{1}{2}}\right]. 
\end{align}
The first term is consistent due to the continuity of $u^n$ in $L^2_{\sigma}$. In the last term we apply (\ref{noise bound 2}) and Young's Inequality to observe that
\begin{align} \nonumber
c\left(\int_0^{t}\sum_{i=1}^\infty\inner{B_i\check{u}^n_s}{\check{u}^n_s}^2ds\right)^{\frac{1}{2}} \leq c\left(\int_0^{t}\norm{\check{u}^n_s}^4ds\right)^{\frac{1}{2}} &\leq c\left(\sup_{r \in [0,t]}\norm{\check{u}^n_r}^2\int_0^{t}\norm{\check{u}^n_s}^2 ds\right)^{\frac{1}{2}}\\ \label{BDG type control}
&\leq \frac{1}{2}\left(\sup_{r \in [0,t]}\norm{\check{u}^n_r}^2\right) + c\left(\int_0^{t}\norm{\check{u}^n_s}^2 ds\right).
\end{align}
From here we deduce the bound
\begin{align}\nonumber\mathbbm{E}\left(\sup_{r \in [0,t]}\norm{\check{u}^{n}_{ r }}^2 + \int_0^{t}\norm{\check{u}^n_s}_1^2ds\right) \leq c\mathbbm{E}\left[\norm{u_0}^2 + \int_0^{t} \norm{\check{u}^n_s}^2 ds\right] 
\end{align}
noting that $c$ is also dependent on the fixed $\nu$, from which the standard Gr\"{o}nwall Inequality yields that
\begin{align}\nonumber\mathbbm{E}\left(\sup_{r \in [0,T]}\norm{\check{u}^{n}_{ r }}^2 + \int_0^{T}\norm{\check{u}^n_s}_1^2ds\right) \leq C\mathbbm{E}\left(\norm{u_0}^2\right).
\end{align}
It only remains to remove the localisation by $\tau^R_n$, to which we appreciate that the sequence $(\tau^R_n)$ is $\mathbbm{P}-a.s.$ increasing to $T$ due to the continuity of $u^n$; therefore we may apply the Monotone Convergence Theorem taking $R \rightarrow \infty$ which gives the result.
\end{proof}

\subsection{Convergence to Zero at the Boundary} \label{subs convergence to zero}

\begin{proposition} \label{prop boundary to zero}
There exists a constant $C$ such that for all $n \in \N$,
\begin{equation} \label{boundary estimate}
    \left[\sup_{\partial \mathscr{O}}(\alpha_n - \kappa)\right]\mathbbm{E}\left(\int_0^T \norm{u^n_s}_{L^2(\partial \mathscr{O};\R^2)}^2ds\right) \leq C\mathbbm{E}\left(\norm{u_0}^2\right)
\end{equation}
and in particular,
  \begin{equation} \label{zero limit}
   \lim_{n \rightarrow \infty} \mathbbm{E}\left(\int_0^T \norm{u^n_s}_{L^2\left(\partial \mathscr{O}; \R^2\right)}^2 ds \right) = 0.   
  \end{equation}

\end{proposition}

\begin{proof}
    We again use the energy equality (\ref{for the uniform estimates}), however now in the boundary integral we use a control
    \begin{align*}\int_0^{r\wedge \tau^R_n}\inner{(\kappa - \alpha_n)u^n_s}{u^n_s}_{L^2(\partial \mathscr{O};\R^2)}ds &\leq  \left[\sup_{ \partial \mathscr{O}}(\kappa - \alpha_n)\right]\int_0^{r\wedge \tau^R_n}\norm{u^n_s}^2_{L^2(\partial \mathscr{O};\R^2)}ds\\ &= -  \left[\inf_{\partial \mathscr{O}}(\alpha_n - \kappa)\right]\int_0^{r\wedge \tau^R_n}\norm{u^n_s}^2_{L^2(\partial \mathscr{O};\R^2)}ds.
    \end{align*}
Taking this term to the left hand side and following the steps of Lemma \ref{lemma uniform}, we arrive at
\begin{align}\nonumber\mathbbm{E}\left(\sup_{r \in [0,t]}\norm{\check{u}^{n}_{ r }}^2 + \int_0^{T}\norm{\check{u}^n_s}_1^2ds + \inf_{\partial \mathscr{O}}(\alpha_n - \kappa)\int_0^{T}\norm{\check{u}^n_s}^2_{L^2(\partial \mathscr{O};\R^2)}ds\right) \leq c\mathbbm{E}\left[\norm{u_0}^2 + \int_0^{T} \norm{\check{u}^n_s}^2 ds\right].
\end{align}
Simply ignoring the first two non-negative terms, and using Lemma \ref{lemma uniform} to control the last term, we obtain
\begin{align}\nonumber \left[\inf_{\partial \mathscr{O}}(\alpha_n - \kappa)\right]\mathbbm{E}\left(\int_0^{T}\norm{\check{u}^n_s}^2_{L^2(\partial \mathscr{O};\R^2)}ds\right) \leq c\mathbbm{E}\left(\norm{u_0}^2\right).
\end{align}
Now we multiply both sides by $\frac{\sup_{\partial \mathscr{O}}(\alpha_n - \kappa)}{\inf_{\partial \mathscr{O}}(\alpha_n - \kappa)}$, and using (\ref{sup bound with kappa}) verify that
\begin{align}\nonumber \left[\sup_{\partial \mathscr{O}}(\alpha_n - \kappa)\right]\mathbbm{E}\left(\int_0^{T}\norm{\check{u}^n_s}^2_{L^2(\partial \mathscr{O};\R^2)}ds\right) \leq c\frac{\sup_{\partial \mathscr{O}}(\alpha_n - \kappa)}{\inf_{\partial \mathscr{O}}(\alpha_n - \kappa)}\mathbbm{E}\left(\norm{u_0}^2\right) \leq C\mathbbm{E}\left(\norm{u_0}^2\right).
\end{align}
By the same application of the Monotone Convergence Theorem as at the end of Lemma \ref{lemma uniform}, we deduce (\ref{boundary estimate}). Then (\ref{zero limit}) follows by dividing both sides by $\sup_{\partial \mathscr{O}}(\alpha_n - \kappa)$ and taking the limit as $n \rightarrow \infty$, using that $\inf_{\partial \mathscr{O}}(\alpha_n)$ hence $\sup_{\partial \mathscr{O}}(\alpha_n - \kappa)$ diverges to infinity.

\end{proof}

\subsection{Tightness} \label{subs tightness}


\begin{proposition} \label{prop for tightness}
    For any $R > 1$,
    \begin{equation} \label{required condition}
    \lim_{\delta \rightarrow 0^+}\sup_{n \in \N}\mathbbm{E}\left(\int_0^{T-\delta}\norm{\hat{u}^{n}_{s + \delta} - \hat{u}^{n}_s}^2ds\right)= 0.
\end{equation} 
Moreover the sequence of the laws of $(u^n)$ is tight in the space of probability measures over\\ $L^2\left([0,T];L^2_{\sigma}\right)$.
\end{proposition}

\begin{proof}
Let us first establish why (\ref{required condition}) is sufficient to deduce the desired tightness. We are looking to apply Lemma \ref{Lemma 5.2} for the spaces $\mathcal{H}_1:= \bar{W}^{1,2}_{\sigma}$, $\mathcal{H}_2:= L^2_{\sigma}$, noting that condition (\ref{first condition}) is satisfied due to Lemma \ref{lemma uniform}. Then (\ref{required condition}) is nearly the remaining (\ref{second condition}), except for the localisation at $\tau^R_n$; due to the uniform estimates from Lemma \ref{lemma uniform} then it is sufficient to show only (\ref{required condition}), which was demonstrated in [\cite{rockner2022well}] Lemma 2.12, [\cite{goodair2024weak}] Subsection 2.4.\\

We now turn to verifying (\ref{required condition}). Similarly to (\ref{for the uniform estimates}), for any fixed $0 \leq s < T$, by considering the evolution equation (in $\delta$) satisfied by the increment, we observe the energy equality
\begin{align}\nonumber\norm{\hat{u}^{n}_{s + \delta} - \hat{u}^{n}_{s}}^2 &= - 2\int_s^{s+\delta}\inner{\mathcal{L}_{\check{u}^n_r}\check{u}^n_r}{\check{u}^n_r - \check{u}^n_s}dr\\ \nonumber &- 2\nu\int_s^{s+\delta}\inner{\check{u}^n_r}{\check{u}^n_r - \check{u}^n_s}_1dr + 2\nu\int_s^{s+\delta}\inner{(\kappa - \alpha_n)\check{u}^n_r}{\check{u}^n_r - \check{u}^n_s}_{L^2(\partial \mathscr{O};\R^2)}dr\\ \nonumber &+ \int_s^{s+\delta} \sum_{i=1}^\infty \inner{B_i\check{u}^n_r}{B_i^*\left(\check{u}^n_r - \check{u}^n_s\right)}dr + \int_s^{s+\delta}\sum_{i=1}^\infty\norm{\mathcal{P}B_i\check{u}^n_r}^2dr\\ &- 2\sum_{i=1}^\infty\int_s^{s+\delta}\inner{B_i\check{u}^n_r}{\check{u}^n_r - \check{u}^n_s}dW^i_r. \label{big expression}
\end{align}
As we look to reduce each term we begin with the nonlinear operator, obtaining that
\begin{align}
    \nonumber \abs{\inner{\mathcal{L}_{\check{u}^n_r}\check{u}^n_r}{\check{u}^n_r - \check{u}^n_s}} = \abs{\inner{\check{u}^n_r}{\mathcal{L}_{\check{u}^n_r}\check{u}^n_s}} \leq \norm{\check{u}^n_r}_{L^4}\norm{\mathcal{L}_{\check{u}^n_r}\check{u}^n_s}_{L^{\frac{4}{3}}} &\leq c\norm{\check{u}^n_r}_{L^4}\norm{\check{u}^n_r}_{L^4}\norm{\check{u}^n_s}_{1}\\ &\leq c \norm{\check{u}^n_r}\norm{\check{u}^n_r}_1\norm{\check{u}^n_s}_{1} \nonumber
\end{align}
having utilised Ladyzhenskaya's Inequality. In the following term, 
$$ -\inner{\check{u}^n_r}{\check{u}^n_r - \check{u}^n_s}_1 \leq \inner{\check{u}^n_r}{\check{u}^n_s}_1 \leq \norm{\check{u}^n_r }_1\norm{\check{u}^n_s }_1$$
and for the boundary integral we have that
\begin{align*}
    \inner{(\kappa - \alpha_n)\check{u}^n_r}{\check{u}^n_r - \check{u}^n_s}_{L^2(\partial \mathscr{O};\R^2)} &\leq \inner{(\alpha_n - \kappa)\check{u}^n_r}{\check{u}^n_s}_{L^2(\partial \mathscr{O};\R^2)}\\ &\leq \left[\sup_{\partial \mathscr{O}}(\alpha_n - \kappa)\right]\norm{\check{u}^n_r}_{L^2(\partial \mathscr{O};\R^2)}\norm{\check{u}^n_s}_{L^2(\partial \mathscr{O};\R^2)}.
\end{align*} 
 In the following noise terms we again use (\ref{noise bound 1}), reaching that
\begin{align*}
    \sum_{i=1}^\infty \left(\inner{B_i\check{u}^n_r}{B_i^*\left(\check{u}^n_r - \check{u}^n_s\right)} +\norm{\mathcal{P}B_i\check{u}^n_r}^2 \right) &\leq \sum_{i=1}^\infty \left(\inner{B_i\check{u}^n_r}{B_i^*\check{u}^n_r} +\norm{B_i\check{u}^n_r}^2  + \inner{B_i\check{u}^n_r}{B_i^*\check{u}^n_s}\right)\\
    &\leq c\left(\norm{\check{u}^n_r}^2 + \norm{\check{u}^n_r}_1\norm{\check{u}^n_s}_1\right).
\end{align*}
Plugging all of these estimates back into (\ref{big expression}), as well as taking expectation which disappears the stochastic integral, we arrive at the inequality
\begin{align}\nonumber\mathbbm{E}\left(\norm{\hat{u}^{n}_{s + \delta} - \hat{u}^{n}_{s}}^2\right) &\leq c\mathbbm{E}\left(\int_s^{s+\delta}\norm{\check{u}^n_r}\norm{\check{u}^n_r}_1\norm{\check{u}^n_s}_{1} +  \norm{\check{u}^n_r}_1\norm{\check{u}^n_s}_1 +
\norm{\check{u}^n_r}^2 dr\right) \\ &+ c\mathbbm{E}\left(\int_s^{s+\delta}\left[\sup_{\partial \mathscr{O}}(\alpha_n - \kappa)\right]\norm{\check{u}^n_r}_{L^2(\partial \mathscr{O};\R^2)}\norm{\check{u}^n_s}_{L^2(\partial \mathscr{O};\R^2)} dr \right). \nonumber
\end{align}
We begin to use the control generated by the stopping time $\tau^R_n$, which prescribes a uniform in $\omega$ bound on $\sup_{r\in[0,T]}\norm{\check{u}^n_r}^2$. Putting this into the top line, allowing our variable constant $c$ to depend on the fixed $R$, integrating over $s \in [0,T-\delta]$ and applying Fubini-Tonelli's Theorem reveals that
\begin{align}\nonumber\mathbbm{E}\left(\int_0^{T-\delta}\norm{\hat{u}^{n}_{s + \delta} - \hat{u}^{n}_{s}}^2ds\right) &\leq c\mathbbm{E}\left(\int_0^{T-\delta}\int_s^{s+\delta}\norm{\check{u}^n_r}_1\norm{\check{u}^n_s}_{1} + 1 \,  drds\right) \\ &+ c\mathbbm{E}\left(\int_0^{T-\delta}\int_s^{s+\delta}\left[\sup_{\partial \mathscr{O}}(\alpha_n - \kappa)\right]\norm{\check{u}^n_r}_{L^2(\partial \mathscr{O};\R^2)}\norm{\check{u}^n_s}_{L^2(\partial \mathscr{O};\R^2)} dr ds\right). \nonumber
\end{align}
In view of what we want to prove, the goal here is to reduce the right hand side terms to ones uniformly bounded in $n$ and which approach zero as $\delta \rightarrow 0^+$. Isolating the constant integral, it is immediate that
$$ c\mathbbm{E}\left(\int_0^{T-\delta}\int_s^{s+\delta} 1 \,  drds\right) \leq c \delta $$
where $c$ is also dependent on $T$. In addition,
\begin{align*}
    \int_s^{s+\delta}\norm{\check{u}^n_r}_1\norm{\check{u}^n_s}_{1}  dr \leq \norm{\check{u}^n_s}_{1}\int_s^{s+\delta}\norm{\check{u}^n_r}_1  dr
    &\leq \norm{\check{u}^n_s}_{1}\left[\left(\int_s^{s+\delta}1 \,  dr\right) ^{\frac{1}{2}}\left(\int_s^{s+\delta}\norm{\check{u}^n_r}_1 ^2 dr\right)^{\frac{1}{2}} \right]\\ &\leq c\delta^{\frac{1}{2}}\norm{\check{u}^n_s}_{1}
\end{align*}
using, on this occasion, the uniform in $\omega$ control on $\int_0^T\norm{\check{u}^n_r}_1 ^2 dr$. The boundary integrals are treated in the same way,
\begin{align*}
    &\int_s^{s+\delta}\left[\sup_{\partial \mathscr{O}}(\alpha_n - \kappa)\right]\norm{\check{u}^n_r}_{L^2(\partial \mathscr{O};\R^2)}\norm{\check{u}^n_s}_{L^2(\partial \mathscr{O};\R^2)}  dr\\ & \qquad \qquad \qquad \qquad \leq \left[\sup_{\partial \mathscr{O}}(\alpha_n - \kappa)\right]^{\frac{1}{2}}\norm{\check{u}^n_s}_{L^2(\partial \mathscr{O};\R^2)}\int_s^{s+\delta}\left[\sup_{\partial \mathscr{O}}(\alpha_n - \kappa)\right]^{\frac{1}{2}}\norm{\check{u}^n_r}_{L^2(\partial \mathscr{O};\R^2)}  dr\\
    & \qquad  \qquad \qquad \qquad \leq c \delta^{\frac{1}{2}}\left[\sup_{\partial \mathscr{O}}(\alpha_n - \kappa)\right]^{\frac{1}{2}}\norm{\check{u}^n_s}_{L^2(\partial \mathscr{O};\R^2)}.
\end{align*}
Altogether, the problem is now reduced to showing that
\begin{align*}
   \lim_{\delta \rightarrow 0^+}\sup_{n \in \N} \delta^{\frac{1}{2}}\mathbbm{E}\left(\int_0^{T-\delta}\norm{\check{u}^n_s}_{1} +   \left[\sup_{\partial \mathscr{O}}(\alpha_n - \kappa)\right]^{\frac{1}{2}}\norm{\check{u}^n_s}_{L^2(\partial \mathscr{O};\R^2)}  ds \right) = 0
\end{align*}
or simply 
\begin{align*}
   \sup_{n \in \N}\mathbbm{E}\left(\int_0^{T}\norm{\check{u}^n_s}_{1} +   \left[\sup_{\partial \mathscr{O}}(\alpha_n - \kappa)\right]^{\frac{1}{2}}\norm{\check{u}^n_s}_{L^2(\partial \mathscr{O};\R^2)}  ds \right) < \infty
\end{align*}
which is clear given the uniform control on the integral of the square.
    
\end{proof}

\begin{proposition} \label{prop for tightness two}
    For any sequence of stopping times $(\gamma_n)$ with $\gamma_n: \Omega \rightarrow [0,T]$, and any  $R > 1$, $\phi \in W^{1,2}_{\sigma}$,
    \begin{equation} \label{required condition primed}
    \lim_{\delta \rightarrow 0^+}\sup_{n \in \N}\mathbbm{E}\left(\left\vert\left\langle \hat{u}^{n}_{(\gamma_n + \delta) \wedge T} - \hat{u}^{n}_{\gamma_n} , \phi \right\rangle \right\vert\right)= 0.
\end{equation} 
Moreover the sequence of the laws of $(u^n)$ is tight in the space of probability measures over\\ $C\left([0,T];W^{-\varepsilon,2}_{\sigma}\right)$.
\end{proposition}

\begin{proof}
Once more we first appreciate that tightness will follow from (\ref{required condition primed}) by applying Lemma \ref{lemma for D tight}, taking $\mathcal{Y}:= W^{\varepsilon,2}_{\sigma}$, $\mathcal{H}:= L^2_{\sigma}$ and $\mathcal{V}:= W^{1,2}_{\sigma}$, due to the uniform estimates of Lemma \ref{lemma uniform} and allowing for localisation exactly as in Proposition \ref{prop for tightness}. Towards (\ref{required condition primed}) we consider the identity satisfied by the increment, yielding
\begin{align}\nonumber\left\vert\inner{\hat{u}^{n}_{(\gamma_n + \delta) \wedge T} - \hat{u}^{n}_{\gamma_n}}{\phi}\right\vert &\leq \int_{\gamma_n}^{(\gamma_n + \delta) \wedge T}\abs{\inner{\mathcal{L}_{\check{u}^n_s}\check{u}^n_s}{\phi}}ds + \nu\int_{\gamma_n}^{(\gamma_n + \delta) \wedge T}\abs{\inner{\check{u}^n_s}{\phi}}_1ds\\ \nonumber &+ \int_{\gamma_n}^{(\gamma_n + \delta) \wedge T} \sum_{i=1}^\infty \abs{\inner{B_i\check{u}^n_s}{B_i^*\phi}}ds + \abs{\sum_{i=1}^\infty\int_{\gamma_n}^{(\gamma_n + \delta) \wedge T}\inner{B_i\check{u}^n_s}{\phi}dW^i_s}. \nonumber
\end{align}
There is no boundary integral in the above due to the choice of $\phi \in W^{1,2}_{\sigma}$, which is imperative for our estimates. The integrands in the time integrals are all controlled exactly as they were in Proposition \ref{prop for tightness}, and absorbing norms of $\phi$ into the constant, this reduces to 
\begin{align}\nonumber\left\vert\inner{\hat{u}^{n}_{(\gamma_n + \delta) \wedge T} - \hat{u}^{n}_{\gamma_n}}{\phi}\right\vert \leq c\int_{\gamma_n}^{(\gamma_n + \delta) \wedge T} 1 + \norm{\check{u}^n_s}_1 ds + \abs{\sum_{i=1}^\infty\int_{\gamma_n}^{(\gamma_n + \delta) \wedge T}\inner{B_i\check{u}^n_s}{\phi}dW^i_s}. \nonumber
\end{align}
Now we take expectation of the above, and applying the Burkholder-Davis-Gundy Inequality to the stochastic integral, we see that
\begin{align*}
   \mathbbm{E}\left( \abs{\sum_{i=1}^\infty\int_{\gamma_n}^{(\gamma_n + \delta) \wedge T}\inner{B_i\check{u}^n_s}{\phi}dW^i_r}\right) \leq c\mathbbm{E}\left( \int_{\gamma_n}^{(\gamma_n + \delta) \wedge T}\sum_{i=1}^\infty\inner{\check{u}^n_s}{B_i^*\phi}^2ds\right)^{\frac{1}{2}} \leq c\delta^{\frac{1}{2}}
\end{align*}
having immediately used that $\norm{\check{u}^n_s} \leq c$ due to the stopping time $\tau^R_n$. Therefore
\begin{align}\nonumber\mathbbm{E}\left(\left\vert\inner{\hat{u}^{n}_{(\gamma_n + \delta) \wedge T} - \hat{u}^{n}_{\gamma_n}}{\phi}\right\vert \right)\leq c\mathbbm{E}\left( \delta^{\frac{1}{2}} + \delta + \int_{\gamma_n}^{(\gamma_n + \delta) \wedge T} \norm{\check{u}^n_s}_1 ds\right).
\end{align}
Exactly as we saw in Proposition \ref{prop for tightness},
$$\int_{\gamma_n}^{(\gamma_n+\delta) \wedge T}\norm{\check{u}^n_s}_1  ds
    \leq \left[\left(\int_{\gamma_n}^{(\gamma_n+\delta) \wedge T}1 \,  ds\right) ^{\frac{1}{2}}\left(\int_{\gamma_n}^{(\gamma_n+\delta) \wedge T}\norm{\check{u}^n_s}_1 ^2 ds\right)^{\frac{1}{2}} \right] \leq c\delta^{\frac{1}{2}}$$
    from which taking the supremum over $n \in \N$ and limit $\delta \rightarrow 0^+$ gives the result.

\end{proof}

\subsection{Identification of the Limit} \label{subs identification of the limit}

With tightness achieved, it is now a standard procedure to apply the Prohorov and Skorohod Representation Theorems to deduce the existence of a new probability space on which a sequence of processes with the same distribution as a subsequence of $(u^{n})$ have some almost sure convergence to a limiting process. For notational simplicity we take this subsequence and keep it simply indexed by $n$. We state the precise result in the below theorem, following e.g. [\cite{nguyen2021nonlinear}] Proposition 4.9 and Theorem 4.10. The additional regularity and moment convergence is obtained exactly as in [\cite{goodair2024weak}] Subsection 2.5.

\begin{proposition} \label{theorem for new prob space}
    There exists a filtered probability space $\left(\tilde{\Omega},\tilde{\mathcal{F}},(\tilde{\mathcal{F}}_t), \tilde{\mathbbm{P}}\right)$, a Cylindrical Brownian Motion $\tilde{\mathcal{W}}$ over $\mathfrak{U}$ with respect to $\left(\tilde{\Omega},\tilde{\mathcal{F}},(\tilde{\mathcal{F}}_t), \tilde{\mathbbm{P}}\right)$, a random variable $ \tilde{u}_0: \tilde{\Omega} \rightarrow L^2_{\sigma}$, and progressively measurable processes $(\tilde{u}^n), \tilde{u} :\tilde{\Omega} \times [0,T] \rightarrow \bar{W}^{1,2}_{\sigma}$ such that:
   \begin{enumerate}        
        \item $\tilde{u}_0$ has the same law as $u_0$;
        \item \label{new item 3} $\tilde{u}^n$ is the unique weak solution of (\ref{NSalphan}) on $\left(\tilde{\Omega},\tilde{\mathcal{F}},(\tilde{\mathcal{F}}_t), \tilde{\mathbbm{P}}, \tilde{\mathcal{W}}\right)$ with initial condition $\tilde{u}_0$;
    \item \label{new item 4}$\tilde{u}^n \rightarrow \tilde{u}$ in $C\left([0,T];W^{-\varepsilon,2}_{\sigma} \right) \cap L^2\left([0,T];L^2_{\sigma} \right)$ $\tilde{\mathbbm{P}}-a.s.$, and in\\ $L^2\left[ \tilde{\Omega};C\left([0,T];W^{-\varepsilon,2}_{\sigma}\right) \cap L^2\left([0,T];L^2_{\sigma} \right)  \right]$; 
    \item \label{new item 5} $\tilde{u} \in L^\infty\left([0,T];L^2_{\sigma} \right) \cap L^2\left([0,T];\bar{W}^{1,2}_{\sigma}\right)$ $\tilde{\mathbbm{P}}-a.s.$ and in\\ $L^2\left[\tilde{\Omega};L^\infty\left([0,T];L^2_{\sigma} \right) \cap L^2\left([0,T];\bar{W}^{1,2}_{\sigma}\right)\right].$
    \end{enumerate}
    \end{proposition}

We wish to show that $\tilde{u}$ is the unique weak solution of (\ref{NSinfty}) on $\left(\tilde{\Omega},\tilde{\mathcal{F}},(\tilde{\mathcal{F}}_t), \tilde{\mathbbm{P}}, \tilde{\mathcal{W}}\right)$. This is the motivation behind the following lemmas.

\begin{lemma} \label{lemma for zero boundary in limit}
    $\tilde{u}$ belongs $\tilde{\mathbbm{P}} \times \lambda-a.s.$ to $W^{1,2}_{\sigma}$. In particular, $\tilde{u}$ is progressively measurable in $W^{1,2}_{\sigma}$ and for $\mathbbm{P}-a.e.$ $\omega$, $\tilde{u}_{\cdot}(\omega) \in L^2\left([0,T];W^{1,2}_{\sigma}\right)$. 
\end{lemma}

\begin{proof}
    The key element of this proof is Proposition \ref{prop boundary to zero}, namely (\ref{zero limit}). Indeed (\ref{zero limit}) shows us that $(\tilde{u}^n)$ converges to zero in $L^2\left[\tilde{\Omega} \times [0,T]; L^2\left(\partial \mathscr{O};\R^2\right)\right]$. To prove the lemma we need to reconcile this convergence with that to $\tilde{u}$. By (\ref{inequality from Lions}) we have that
    \begin{align*}
        \tilde{\mathbbm{E}}\int_0^T\norm{\tilde{u}^n_s - \tilde{u}_s}_{L^2\left(\partial \mathscr{O};\R^2\right)}^2 ds  &\leq \tilde{\mathbbm{E}}\int_0^T\norm{\tilde{u}^n_s - \tilde{u}_s}\norm{\tilde{u}^n_s - \tilde{u}_s}_{1} ds\\
        &\leq \left(\tilde{\mathbbm{E}}\int_0^T\norm{\tilde{u}^n_s - \tilde{u}_s}^2 ds\right)^{\frac{1}{2}}\left(\tilde{\mathbbm{E}}\int_0^T\norm{\tilde{u}^n_s - \tilde{u}_s}_{1}^2 ds\right)^{\frac{1}{2}}\\
        &\leq 2\left(\tilde{\mathbbm{E}}\int_0^T\norm{\tilde{u}^n_s - \tilde{u}_s}^2 ds\right)^{\frac{1}{2}}\left(\tilde{\mathbbm{E}}\int_0^T\norm{\tilde{u}^n_s}_1^2 + \norm{\tilde{u}_s}_{1}^2 ds\right)^{\frac{1}{2}}
    \end{align*}
which approaches zero as $n \rightarrow \infty$, due to the convergence in item \ref{new item 4} and the uniform bounds from Lemma \ref{lemma uniform}. So $(\tilde{u}^n)$ is convergent to both zero and $\tilde{u}$ in $L^2\left[\tilde{\Omega} \times [0,T]; L^2\left(\partial \mathscr{O};\R^2\right)\right]$, so by uniqueness of limits we must have that $\tilde{u}$ is equal to zero in this topology. Therefore $\tilde{u}$ is $\tilde{\mathbbm{P}} \times \lambda-a.s.$ equal to zero in $L^2\left(\partial \mathscr{O};\R^2\right)$ thus has null trace, and as it is already established to belong to $\bar{W}^{1,2}_{\sigma}$ then it is in $W^{1,2}_{\sigma}$. The remaining assertions of the lemma follow from the corresponding properties in $\bar{W}^{1,2}_{\sigma}$ and the fact that the topology on $W^{1,2}_{\sigma}$ is the restriction of the topology on $\bar{W}^{1,2}_{\sigma}$.\footnote{Recall that progressive measurability in the solution theory is of a $\mathbbm{P} \times \lambda-a.s.$ equal version of the process, see [\cite{goodair2024stochastic}] Remark 3.1. }

\end{proof}

\begin{lemma} \label{lemma for limit identity}
    $\tilde{u}$ satisfies the identity
 \begin{align} \nonumber
     \inner{\tilde{u}_t}{\phi} = \inner{\tilde{u}_0}{\phi} - \int_0^{t}\inner{\mathcal{L}_{\tilde{u}_s}\tilde{u}_s}{\phi}ds &- \nu\int_0^{t} \inner{\tilde{u}_s}{\phi}_1 ds\\ &+ \frac{1}{2}\int_0^{t}\sum_{i=1}^\infty \inner{B_i\tilde{u}_s}{B_i^*\phi} ds - \int_0^{t} \inner{B\tilde{u}_s}{\phi} d\mathcal{W}_s \nonumber
\end{align}
    for every $\phi \in W^{1,2}_{\sigma}$, $\tilde{\mathbbm{P}}-a.s.$ in $\R$ for all $t\in[0,T]$. Moreover for $\tilde{\mathbbm{P}}-a.e.$ $\omega$, $\tilde{u}_{\cdot}(\omega) \in C\left([0,T];L^2_{\sigma} \right)$.
\end{lemma}

\begin{proof}
We show that $\tilde{u}$ satisfies this identity by showing the convergence $\tilde{\mathbbm{P}}-a.s.$ of a further subsequence, term by term, of the identity satisfied by $(\tilde{u}^n)$. This argument was given for an abstract SPDE in [\cite{goodair2024weak}] Proposition 2.18 which was explicitly shown to admit the Navier-Stokes equations with transport-stretching noise as a special case, containing the details for the convergence in each term. We only draw particular attention to the boundary integral in the identity satisfied by $\tilde{u}^n$, which is immediately zero as we are testing with $\phi \in W^{1,2}_{\sigma}$ and not a general $\psi \in \bar{W}^{1,2}_{\sigma}$. With the identity settled, the continuity follows as an immediate application of Proposition \ref{Ito formula} for $\mathcal{H}_1 = W^{1,2}_{\sigma}$, $\mathcal{H}_2:= L^2_{\sigma}$, $\mathcal{H}_3:= W^{-1,2}_{\sigma}$.

\end{proof}

As a consequence of Lemmas \ref{lemma for zero boundary in limit} and \ref{lemma for limit identity}, $\tilde{u}$ is the unique weak solution of (\ref{NSinfty}) on\\ $\left(\tilde{\Omega},\tilde{\mathcal{F}},(\tilde{\mathcal{F}}_t), \tilde{\mathbbm{P}}, \tilde{\mathcal{W}}\right)$ for the initial condition $\tilde{u}_0$.

\subsection{Final Steps of Theorem \ref{main theorem}} \label{subs final steps}

We will prove Theorem \ref{main theorem} by combining the analysis given in Subsection \ref{subs identification of the limit} with Lemma \ref{gyongykrylov} in the appendix. We wish to apply Lemma \ref{gyongykrylov} for the sequence $(\sy^n) := (u^n)$ and $\mathcal{Y}:= C\left([0,T];W^{-\varepsilon,2}_{\sigma} \right) \cap L^2\left([0,T];L^2_{\sigma} \right)$. To this end we consider any two subsequences $(u^l)$, $(u^m)$ and the product $(u^l,u^m)$. As the whole sequence $(u^n)$ is tight in $\mathcal{Y}$ then so is every subsequence, and as the product of tight sequences is tight then $(u^l,u^m)$ is tight in $\mathcal{Y} \times \mathcal{Y}$. Now we can apply the Prohorov and Skorohod Representation Theorems for this product, yielding a $\tilde{\mathbbm{P}}-a.s.$ convergent subsequence $(\tilde{u}^{l_k},\tilde{u}^{m_k})$ in $\mathcal{Y} \times \mathcal{Y}$ with all of the properties of Proposition \ref{theorem for new prob space} simply for the product. $\tilde{\mathbbm{P}}-a.e.$ convergence of the product $(\tilde{u}^{l_k},\tilde{u}^{m_k})$ implies the convergence of each individual subsequence, and exactly as was shown in Subsection \ref{subs identification of the limit} each limit must be the unique weak solution of (\ref{NSinfty}) on this probability space, $\tilde{u}$. Therefore $(\tilde{u}^{l_k},\tilde{u}^{m_k})$ converges $\tilde{\mathbbm{P}}-a.s.$ to $(\tilde{u},\tilde{u})$ in $\mathcal{Y} \times \mathcal{Y}$, hence in law as well; by design $(\tilde{u}^{l_k},\tilde{u}^{m_k})$ has the same law as $(u^{l_k},u^{m_k})$, and as $(\tilde{u},\tilde{u})$ is supported only on the diagonal then the hypothesis of Lemma \ref{gyongykrylov} is verified and we can conclude convergence in probability of the whole sequence $(u^n)$ in $C\left([0,T];W^{-\varepsilon,2}_{\sigma} \right) \cap L^2\left([0,T];L^2_{\sigma} \right)$. To prove Theorem \ref{main theorem} we now only need to show that this limit in probability is indeed $u$, which is true as convergence in probability implies the convergence of a subsequence $\mathbbm{P}-a.s.$ to the same limit, which has to be $u$ as already discussed. Therefore, Theorem \ref{main theorem} is proven.




\appendix
\section{Appendix} \label{appendix}

We collect useful results from the literature that have been used throughout the paper.

\begin{proposition} \label{Ito formula}
Let $\mathcal{H}_1 \subset \mathcal{H}_2 \subset \mathcal{H}_3$ be a triplet of embedded Hilbert Spaces where $\mathcal{H}_1$ is dense in $\mathcal{H}_2$, with the property that there exists a continuous nondegenerate bilinear form $\inner{\cdot}{\cdot}_{\mathcal{H}_3 \times \mathcal{H}_1}: \mathcal{H}_3 \times \mathcal{H}_1 \rightarrow \R$ such that for $\phi \in \mathcal{H}_2$ and $\psi \in \mathcal{H}_1$, $$\inner{\phi}{\psi}_{\mathcal{H}_3 \times \mathcal{H}_1} = \inner{\phi}{\psi}_{\mathcal{H}_2}.$$ Suppose that for some $T > 0$ and stopping time $\tau$,
\begin{enumerate}
        \item $\sy_0:\Omega \rightarrow \mathcal{H}_2$ is $\mathcal{F}_0-$measurable;
        \item $f:\Omega \times [0,T] \rightarrow \mathcal{H}_3$ is such that for $\mathbbm{P}-a.e.$ $\omega$, $f(\omega) \in L^2([0,T];\mathcal{H}_3)$;
        \item $B:\Omega \times [0,T] \rightarrow \mathscr{L}^2(\mathfrak{U};\mathcal{H}_2)$ is progressively measurable and such that for $\mathbbm{P}-a.e.$ $\omega$, $B(\omega) \in L^2\left([0,T];\mathscr{L}^2(\mathfrak{U};\mathcal{H}_2)\right)$;
        \item  \label{4*} $\sy:\Omega \times [0,T] \rightarrow \mathcal{H}_1$ is such that for $\mathbbm{P}-a.e.$ $\omega$, $\sy_{\cdot}(\omega)\mathbbm{1}_{\cdot \leq \tau(\omega)} \in L^2([0,T];\mathcal{H}_1)$ and $\sy_{\cdot}\mathbbm{1}_{\cdot \leq \tau}$ is progressively measurable in $\mathcal{H}_1$;
        \item \label{item 5 again*} The identity
        \begin{equation} \label{newest identity*}
            \sy_t = \sy_0 + \int_0^{t \wedge \tau}f_sds + \int_0^{t \wedge \tau}B_s d\mathcal{W}_s
        \end{equation}
        holds $\mathbbm{P}-a.s.$ in $\mathcal{H}_3$ for all $t \in [0,T]$.
    \end{enumerate}
The the equality 
  \begin{align} \label{ito big dog*}\norm{\sy_t}^2_{\mathcal{H}_2} = \norm{\sy_0}^2_{\mathcal{H}_2} + \int_0^{t\wedge \tau} \bigg( 2\inner{f_s}{\sy_s}_{\mathcal{H}_3 \times \mathcal{H}_1} + \norm{B_s}^2_{\mathscr{L}^2(\mathfrak{U};\mathcal{H}_2)}\bigg)ds + 2\int_0^{t \wedge \tau}\inner{B_s}{\sy_s}_{\mathcal{H}_2}d\mathcal{W}_s\end{align}
  holds for any $t \in [0,T]$, $\mathbbm{P}-a.s.$ in $\R$. Moreover for $\mathbbm{P}-a.e.$ $\omega$, $\sy_{\cdot}(\omega) \in C([0,T];\mathcal{H}_2)$. 
\end{proposition}

\begin{proof}
    See [\cite{goodair2024stochastic}] Proposition 4.3, a slight extension of [\cite{prevot2007concise}] Lemma 4.2.5.
\end{proof}

\begin{lemma} \label{Lemma 5.2}
    Let $\mathcal{H}_1, \mathcal{H}_2$ be Hilbert Spaces such that $\mathcal{H}_1$ is compactly embedded into $\mathcal{H}_2$, and for some fixed $T>0$ let $(\sy^n): \Omega \times [0,T] \rightarrow \mathcal{H}_1$ be a sequence of measurable processes such that \begin{equation} \label{first condition} \sup_{n\in \N}\mathbbm{E}\int_0^T\norm{\sy^n_s}^2_{\mathcal{H}_1}ds < \infty\end{equation} and for any $\varepsilon > 0$, 
    \begin{equation}\label{second condition} \lim_{\delta \rightarrow 0^+}\sup_{n \in \N}\mathbbm{P}\left(\left\{\omega \in \Omega:\int_0^{T-\delta}\norm{\sy^n_{s + \delta}(\omega) - \sy^n_s(\omega)}^2_{\mathcal{H}_2}ds > \varepsilon\right\} \right) =0.\end{equation}
    Then the sequence of the laws of $(\sy^n)$ is tight in the space of probability measures over $L^2\left([0,T];\mathcal{H}_2\right)$.
\end{lemma}

\begin{proof}
    See [\cite{rockner2022well}] Lemma 5.2.
\end{proof}

\begin{lemma} \label{lemma for D tight}
    Let $\mathcal{Y}$ be a reflexive separable Banach Space and $\mathcal{H}$ a separable Hilbert Space such that $\mathcal{Y}$ is compactly embedded into $\mathcal{\mathcal{H}}$, and consider the induced Gelfand Triple
    $$\mathcal{Y} \xhookrightarrow{} \mathcal{H} \xhookrightarrow{} \mathcal{Y}^*. $$ For some fixed $T>0$ let $\sy^n: \Omega \rightarrow C\left([0,T];\mathcal{H}\right)$ be a sequence of measurable processes such that for every $t\in[0,T]$, \begin{equation} \label{first condition primed}
        \sup_{n \in \N}\mathbbm{E}\left(\sup_{t\in[0,T]}\norm{\sy^n_t}_{\mathcal{H}}\right) < \infty
    \end{equation}
    and for any sequence of stopping times $(\gamma_n)$ with $\gamma_n: \Omega \rightarrow [0,T]$, any $\varepsilon > 0$, and any $v \in \mathcal{V}$ where $\mathcal{V}$ is a dense set in $\mathcal{Y}$,
    \begin{equation} \label{second condition primed}
        \lim_{\delta \rightarrow 0^+}\sup_{n \in \N}\mathbbm{P}\left(\left\{
    \omega \in \Omega: \left\vert \left\langle \sy^n_{(\gamma_n + \delta) \wedge T} -\sy^n_{\gamma_n }   , v     \right\rangle_{\mathcal{H}} \right\vert > \varepsilon \right\}\right)  = 0.
    \end{equation}
    Then the sequence of the laws of $(\sy^n)$ is tight in the space of probability measures over $C\left([0,T];\mathcal{Y}^*\right)$.
\end{lemma}

\begin{proof}
    This is essentially [\cite{goodair2024stochastic}] Lemma 4.1, slightly loosened as we have replaced taking any $y \in \mathcal{Y}$ with any $v \in \mathcal{V}$. This is justified similarly to [\cite{goodair2024stochastic}] Lemma 4.2 where elements from any dense set are allowed; in modifying the proof we are only concerned with whether the new set $\mathbbm{F}$ defined as the collection of mappings given for every $v \in \mathcal{V}$ on $\mathcal{Y}^*$ by $\phi \mapsto \inner{\phi}{v}_{\mathcal{Y}^* \times \mathcal{Y}}$ separates points in $\mathcal{Y}^*$. This is true as from the reflexivity of $\mathcal{Y}$, the canonical embedding of $\mathcal{V}$ into $\left(\mathcal{Y}^*\right)^*$ is dense so separates points according to $\left(\mathcal{Y}^*\right)^*$. 
\end{proof}

\begin{lemma} \label{gyongykrylov}
    Let $(\sy^n)$ be a sequence of random variables in a Polish Space $\mathcal{Y}$ equipped with Borel $\sigma-$algebra. Then $(\sy^n)$ converges in probability to some $\mathcal{Y}-$valued random variable $\sy$ if and only if for every pair of subsequences $(\sy^l)$ and $(\sy^m)$, the product $(\sy^{l}, \sy^{m})$ admits a further subsequence converging in law in $\mathcal{Y} \times \mathcal{Y}$ to some $\mathcal{Y} \times \mathcal{Y}-$valued random variable supported only on the diagonal. 
\end{lemma}

\begin{proof}
    See [\cite{gyongy1996existence}] Lemma 1.1.
\end{proof}

\addcontentsline{toc}{section}{References}
\bibliographystyle{newthing}
\bibliography{myBibby}

\end{document}